\documentclass[a4paper,abstracton]{scrartcl}

\usepackage{amsfonts,amsthm,amssymb,amsmath}
\usepackage[latin1]{inputenc}
\usepackage[T1]{fontenc}
\usepackage{graphicx}
\usepackage{enumitem}

\newtheorem{thm}{Theorem}[section]
\newtheorem{lem}[thm]{Lemma}
\newtheorem{prp}[thm]{Proposition}
\newtheorem{cor}[thm]{Corollary}

\theoremstyle{remark}

\newtheorem{rmk}[thm]{Remark}
\newtheorem{con}[thm]{Conjecture}

\newcommand{\Aut}{\operatorname{Aut}}
\newcommand{\id}{\operatorname{id}}
\newcommand{\conf}{\operatorname{conf}}
\newcommand{\dist}{\operatorname{\delta}}
\newcommand{\prob}{\mathbb P}
\newcommand{\haar}{\mathbb H}
\newcommand{\sph}[1]{S_{v_0}\! \left( #1 \right)}
\newcommand{\bigO}{\mathcal O}
\newcommand{\sq}{\mathbin \square}

\title{Random colorings and automorphism breaking in locally finite graphs}
\author{Florian Lehner\thanks{The author acknowledges the support of the Austrian Science Fund (FWF), project W1230-N13.}}

\begin{document}

\maketitle

\begin{abstract}
A colouring of a graph $G$ is called distinguishing if its stabiliser in $\Aut G$ is trivial. It has been conjectured that, if every automorphism of a locally finite graph moves infinitely many vertices, then there is a distinguishing $2$-colouring. We study properties of random $2$-colourings of locally finite graphs and show that the stabiliser of such a colouring  is almost surely nowhere dense in $\Aut G$ and a null set with respect to the Haar measure on the automorphism group. We also investigate random $2$-colourings in several classes of locally finite graphs where the existence of a distinguishing $2$-colouring has already been established. It turns out that in all of these cases a random $2$-colouring is almost surely distinguishing.

\end{abstract}

\section{Introduction}
\label{sec:intro}

A colouring of the vertices of a graph $G$ is called distinguishing if it is not preserved by any non-trivial automorphism of $G$. The notion has been introduced by Albertson and Collins \cite{MR1394549}, but problems involving distinguishing colourings have been around much longer. A classic example is Rubin's key problem \cite{rubin} where a blind professor wants to distinguish his keys by the shape of their handles.

While a distinguishing colouring clearly exists for every graph (simply colour every vertex with a different colour), finding a distinguishing colouring with the minimum number of colours can be challenging.

In this paper we focus on infinite, locally finite graphs. For this class of graphs one of the most intriguing questions is whether or not the following conjecture of Tucker \cite{MR2776826} is true, which generalises a result on finite graphs due to Russel and Sundaram \cite{MR1617449}.

\begin{con}
\label{con:tucker}
Let $G$ be an infinite, connected, locally finite graph with infinite motion. Then there is a distinguishing $2$-colouring of $G$.
\end{con}

The conjecture is known to be true for many classes of infinite graphs including trees \cite{MR2302536}, tree-like graphs \cite{MR2302543}, and graphs with countable automorphism group \cite{istw}. In \cite{smtuwa} it is shown that graphs satisfying the so-called distinct spheres condition have infinite motion as well as distinguishing number two. Examples for such graphs include leafless trees, graphs with infinite diameter and primitive automorphism group, vertex-transitive graphs of connectivity $1$, and Cartesian products of graphs where at least two factors have infinite diameter. It is also known that Conjecture~\ref{con:tucker} is true for graphs fulfilling certain growth conditions \cite{growth}.

All the results mentioned above were attained by deterministically colouring vertices in order to break certain automorphisms. In the present paper we pursue a different approach. We investigate how random colourings behave with respect to automorphism breaking. The idea suggests itself, especially since the result of Russel and Sundaram can be proved using the probabilistic method. As it turns out, in all of the above examples a random colouring will be almost surely distinguishing which leads to the following conjecture.

\begin{con}
\label{con:random}
Let $G$ be a locally finite graph with infinite motion, then a random colouring of $G$ is almost surely distinguishing.
\end{con}

While we are not able to prove this conjecture, we will show that a random colouring is ``almost'' distinguishing in the following sense.

There is a rather natural topology on the automorphism group of a graph (or more generally, on any group of permutations of a countable set) called the permutation topology. Using this topology and a corresponding Haar measure on the automorphism group we show that the stabiliser of a random colouring is almost surely sparse in at least two ways.

\begin{itemize}
\item It will almost surely be nowhere dense and
\item it will almost surely be a null set with respect to the Haar measure.
\end{itemize}
These properties can also be observed in the slightly more general setting of closed, subdegree finite permutation groups of a countable set.

The rest of the paper is structured as follows. Section~\ref{sec:notions} contains all necessary notions and notations. We then define a family of ultrametrics on a group of permutations of a countable set. The induced topology of each member of this family will be the permutation topology mentioned earlier, hence by studying those metrics we will gain some insight into properties of this topology. We are particularly interested in properties of subdegree finite permutation groups which are generalisations of automorphism groups of locally finite graphs and will be the topic of Sections~\ref{sec:stabiliser} and \ref{sec:randomcolour}. In Section~\ref{sec:stabiliser} we will study topological properties of stabilisers of colourings, partial colourings, and subsets of the set on which the permutation group acts. In Section~\ref{sec:randomcolour} we investigate properties of the stabiliser of a random colouring. This section also contains the proofs of the sparsity results mentioned above. Finally, in Section~\ref{sec:graphcolour}, we turn to random colourings of locally finite graphs. We verify Conjecture~\ref{con:random} for many classes of graphs and show that the requirement of local finiteness is necessary by giving a non locally finite counterexample.

\section{Notions and notations}
\label{sec:notions}

Throughout most of this paper we will use Greek letters for group related variables while the Latin alphabet will be reserved for sets on which the group acts. Furthermore these sets  will usually be countable although one could possibly extend some of the results to uncountable sets as well.

Let $S$ be a set and let $\Gamma$ be a group acting on $S$ from the left (all definitions apply analogously to right actions). The image of a point $s \in S$ under an element $\gamma \in \Gamma$ is denoted by $\gamma s$. If $\Delta$ is a subset of $\Gamma$ we denote by $\Delta s =\{\gamma s \mid \gamma \in \Delta\}$ the \emph{orbit of $s$ under $\Delta$}. If $\Delta$ is a subgroup of $\Gamma$ it is well known that $s \in \Delta t$ if and only if $t \in \Delta s$. 

For a point $s \in S$ the \emph{stabiliser of $s$ in $\Gamma$} is defined as $\Gamma_s = \{\gamma \in \Gamma \mid \gamma s = s\}$ and it is well known that this is a subgroup of $\Gamma$. We say that $\Gamma$ is \emph{subdegree finite} if for every $s \in S$ all orbits under $\Gamma_s$ are finite. If $S' \subseteq S$ then we denote by $\Gamma _{S'}$ the \emph{setwise stabiliser of $S'$ in $\Gamma$}, that is, $\Gamma_{S'} = \{\gamma \in \Gamma \mid \forall s \in S' \colon \gamma s \in S'\}$. The \emph{pointwise stabiliser of $S$ in $\Gamma$} is the set $\Gamma_{(S')} = \bigcap_{s \in S'} \Gamma_s$.

The action of $\Gamma$ is \emph{faithful} if different group elements act by different permutations on $S$. In this case we will not distinguish between $\gamma \in \Gamma$ and the corresponding permutation of $S$. Hence we will view $\Gamma$ as a group of permutations of $S$, that is, $\Gamma$ is seen as a subgroup of the group $\Pi_S$ of all bijections $S \to S$.

For the rest of this section let $S$ be a countable set and let $\Gamma$ be a group of permutations of $S$. We say that $\Gamma$ is \emph{closed} if it is a closed subgroup of $\Pi_S$ in the topology of pointwise convergence where $S$ is equipped with the discrete topology. This topology coincides with the permutation topology which we will introduce in the next section.

The \emph{motion} of an element $\gamma \in \Gamma$ is the number of elements of $S$ which are not fixed by $\gamma$. The \emph{motion of the group $\Gamma$} is the minimal motion of a non-trivial element of $\Gamma$. Similarly define the motion of a subset $\Delta \subseteq \Gamma$. Notice that the motion is not necessarily finite, in fact we will mostly be concerned with groups with infinite motion.  Usually $S$ will be the vertex set of a locally finite graph. In this case we define the \emph{motion of the graph $G$} as the motion of $\Gamma = \Aut G$ acting on the vertex set.

Let $C$ be a set. A \emph{$C$-colouring of $S$} is a map $c \colon S \to C$. A \emph{partial $C$-colouring of $S$} is a map $c' \colon S' \to C$ where $S' \subset S$. The set $C$ is referred to as the set of \emph{colours}. Usually $C$ will be the set $\{0,1\}$. In this case we will speak of a \emph{$2$-colouring of $S$}. By a \emph{(partial) $C$-colouring of a graph} $G$ we mean a (partial) $C$-colouring of the vertices of $G$. We denote by $\mathcal C(S,C)$ the set of all $C$-colourings of $S$ and by $\mathcal C(S',C)$ the set of all partial $C$-colourings with domain $S'$. Furthermore let $\mathcal C(S) = \mathcal C(S,\{0,1\})$ and $\mathcal C(S') = \mathcal C(S',\{0,1\})$.

There is a natural (right) action of a group $\Gamma$ of permutations of $S$ on the set of $C$-colourings of $S$ defined by $(c\gamma) (s) = c (\gamma s)$. Notice that even if the action on $S$ is assumed to be faithful this need not necessarily hold for the action on $\mathcal C(S,C)$. 

Given a colouring $c$ and $\gamma \in \Gamma$ we say that \emph{$\gamma$ preserves $c$} if $\gamma$ lies in the stabiliser subgroup $\Gamma _c = \{\gamma \in \Gamma \mid c  \gamma = c\}$. We say that \emph{$\gamma$ preserves a partial colouring $c' \colon S' \to C$} if there are colourings $c_1$ and $c_2$ such that $c_1(s)=c_2(s) = c'(s)$ for every $s \in S'$ and $c_1 \gamma = c_2$. Notice that in general $c_1 \neq c_2$. The \emph{stabiliser $\Gamma_{c'}$ of a partial colouring $c'$} consists of all permutations $\gamma \in \Gamma$ which preserve $c'$. Observe that the stabiliser of a partial colouring need not necessarily be a subgroup of $\Gamma$. If $\gamma$ does not preserve a (partial) colouring $c$ then we say that \emph{$c$ breaks $\gamma$}. We say that \emph{$c$ breaks $\Delta \subseteq \Gamma$} if it breaks every nontrivial element of $\Delta$. A (partial) colouring $c$ which breaks $\Delta$ is called \emph{$\Delta$-distinguishing}. A (partial) colouring of a graph $G$ is called \emph{distinguishing} if it is $(\Aut G)$-distinguishing. Finally we say that a (partial) coloring $c$ \emph{fixes a set $S' \subseteq S$ setwise} if $\Gamma_c \subseteq \Gamma_{S'}$ and that it \emph{fixes $S'$ pointwise} if  $\Gamma_c \subseteq \Gamma_{(S')}$.

\section{A metric on the automorphism group}
\label{sec:metric}

In this section we will describe a family of metrics on a group $\Gamma$ of permutations of a countable set $S$ and discuss some of the properties the induced topology has. The way the metrics are constructed will seem familiar to many readers.  In fact, the construction is similar to the construction of the $p$-adic norm and a similar approach can also be used to equip the end space of a locally finite graph with a metric. It turns out that every metric in this family induces the same topology on $\Gamma$, the so called permutation topology. This topology was first studied in the 1950s by Karass and Solitar \cite{MR0081274} and Maurer \cite{MR0073600} and is a rather natural topology for groups of permutations. As mentioned earlier, another way of introducing the same topology is to equip the set $S$ with the discrete topology and consider the topology of pointwise convergence on $\Gamma$. The paper \cite{moller} by Möller gives a good overview on the permutation topology on closed, subdegree finite permutation groups.

For the construction of the metric, let $S$ be a countable set and let $\Gamma$ be a group of permutations of $S$. Let $(S_i)_{i \in \mathbb N}$ be a sequence of finite subsets of  $S$ such that $S_i \subset S_{i+1}$ and $\lim_{i \to \infty} S_i =S$. For two permutations $\gamma_1, \gamma_2 \in \Gamma$ define the confluent of $\gamma_1$ and $\gamma_2$ as
\[
	\conf(\gamma_1, \gamma_2) = \min \{i \in \mathbb N \mid \exists s \in S_i \colon \gamma_1 \gamma_2^{-1} s \neq s \}-1,
\]
that is, the confluent is the maxiumum $i$ such that $\gamma_1$ and $\gamma_2$ coincide on $S_i$ and it is zero if they differ on $S_1$. Notice that the value of $\conf(\gamma_1, \gamma_2)$ clearly depends on the choice of the sequence $S_i$. 

Now define the distance between $\gamma_1$ and $\gamma_2$ as
\[
\dist(\gamma_1, \gamma_2) = \begin{cases} 0 & \text{if } \gamma_1 = \gamma_2, \\ 2^{-\conf(\gamma_1,\gamma_2)} &\text{otherwise.}\end{cases}
\]
The following proposition shows that the term distance is justified. In fact, $\dist$ even satisfies an ultrametric triangle inequality. As we mentioned earlier the topology induced by $\delta$ does not depend on the choice of the sequence $S_i$. 

\begin{prp}
The function $\dist$ as defined above is an ultrametric on $\Gamma$, all such metrics induce the same topology on $\Gamma$, which makes $\Gamma$ a topological group.
\end{prp}

\begin{proof}
It is readily verified that $\dist(\gamma_1, \gamma_2)$ is symmetric, non-negative, and zero if and only if $\gamma_1 = \gamma_2$. Furthermore, if $r = \min \{ \conf (\gamma_1, \gamma_2) ,\conf (\gamma_2, \gamma_3) \}$ then both $\gamma_1 \gamma_2^{-1}$ and $\gamma _2  \gamma_3^{-1}$ fix $S_r$ pointwise and hence wo does $\gamma_1 \gamma _2^{-1} \gamma _2 \gamma_3^{-1} = \gamma_1 \gamma_3^{-1}$. Thus
\[
	\dist(\gamma_1,\gamma_3) \leq 2^{-r} = \max(\dist(\gamma_1,\gamma_2), \dist(\gamma_2,\gamma_3)),
\]
so $\dist$ is an ultrametric.

Clearly, every sequence $S_i$ induces a different metric on $\Gamma$ but we claim that all of them induce the same topology. 

Indeed, let $\Delta$ be an open neighbourhood of a permutation $\gamma \in \Gamma$ in the topology which comes from the distance $\dist$ defined using the sequence $(S_i)_{i \in \mathbb N}$. Then there is a natural number $n$ such that $\Delta$ contains a $\dist$-ball with center $\gamma$ and radius $2^{-n}$. This implies that $\Delta$ contains all automorphisms $\gamma'$ such that $\gamma \gamma'^{-1}$ fixes $S_n$ pointwise. 

Now consider a different sequence $(S_i')_{i \in \mathbb N}$ of finite subsets of $S$ whose limit is $S$ and use this sequence do define another metric $\dist '$. Then there is an index $m$ such that $S_n \subset S'_m$. So if a permutation $\gamma'$ fulfills $\dist'(\gamma, \gamma') \leq 2^{-m}$ then it certainly holds that $\dist(\gamma, \gamma') \leq 2^{-n}$. In other words, $\Delta$ contains a $\dist'$-ball with center $\gamma$ and radius $2^{-m}$.

So we have proved that an open set with respect to the metric $\dist$ is also open with respect to the metric $\dist'$. Since the converse can be shown in a completely analogous way we conclude that the respective topologies must coincide.

Finally, it is easy to see, that this topology makes $\Gamma$ a topological group. Simply notice that left and right multiplication as well as taking inverses are isometries.
\end{proof}

It is a well known fact that in an ultrametric space distinct balls are disjoint. From this it follows that for any ball $\Delta$ with radius $\varrho$ subballs of $\Delta$ with radius $\varrho' < \varrho$ form a partition of $\Delta$. The following lemma states that this partition will be countable if we partition the whole space, and finite if $\Gamma$ is subdegree finite and $\Delta$ is a strict subset of $\Gamma$.

\begin{lem}
\label{lem:fewballs}
There are only countably many distinct balls of radius $\varrho < 1$ in $\Gamma$. If $\Gamma$ is subdegree finite, then each ball of radius $\varrho < 1$ only has finitely many distinct subballs of radius $\varrho' < \varrho$.
\end{lem}

\begin{proof}
By the definition of $\dist$, balls of radius $\varrho$ are exactly the cosets with respect to the pointwise stabiliser of $S_i$ where $i$ is the unique natural number such that $2^{-i+1}> \varrho \geq 2^{-i}$. Since $S_i$ is finite there are only countably many possibilities to choose the image of $S_i$. So the set of cosets---and hence also the set of balls with radius $\varrho$---is at most countable.

Now let $\Delta \subseteq \Gamma$ be a ball of radius $\varrho < 1$. Since multiplication by a group element is an isometry we may without loss of generality assume that the center of $\Delta$ is $\id$. This implies that $\Delta$ is the pointwise stabiliser of $S_i$ where $2^{-i+1}> \varrho \geq 2^{-i}$. 

A subball of $\Delta$ with radius $\varrho'$ is a coset of $\Delta$ with respect to the stabiliser of $S_j$ where $j$ is the unique natural number such that $2^{-j+1}> \varrho' \geq 2^{-j}$. Hence it suffices to show that there is only a finite number of such cosets.

To see that this is the case notice that every automorphism in $\Delta$ fixes $S_1$.  Furthermore note that $\Gamma$ is subdegree finite, hence the orbit of each $s \in S$ under $\Delta$ is finite. Since $S_j$ is finite  there are only finitely many possibilities to choose an image of $S_j$.
\end{proof}

Now we can use the previous lemma to show that small balls in a closed, subdegree finite permutation group $\Gamma$ are compact. From this result we can derive a multitude of topological properties of $\Gamma$.

\begin{lem}
\label{lem:compactballs}
If $\Gamma$ is closed and subdegree finite then $\Gamma$ is locally compact, more specifically, balls of radius $\varrho < 1 $ are compact.
\end{lem}

\begin{proof}
Since in a metric space compactness and sequential compactness are equivalent it suffices to show that every sequence has a convergent subsequence. So assume we have a sequence $(\gamma_i)_{i \in \mathbb N}$ of pairwise different permutations all of which lie inside a ball $\Delta$ of radius $\varrho < 1$. 

Let $k_0 \in \mathbb N$ such that $2^{-k_0} < \varrho $. Then by Lemma~\ref{lem:fewballs} $\Delta$ has only finitely many subballs of radius $2^{-k_0}$ and hence we can find an infinite subsequence of $\gamma_i$ which is completely contained in one of the subballs $\Delta_0$, say.

The ball $\Delta_0$ again has only finitely many subballs of radius $2^{-k_0-1}$ so we can find a sub-subsequence which lies completely in a subball $\Delta_1$ of $\Delta_0$. Proceeding inductively we obtain a sequence of nested balls $(\Delta_k)_{k \in \mathbb N}$ in $\Gamma$ where the radius of $\Delta_k$ is $2^{-k_0-k}$.

Now we define a permutation $\gamma$ as follows: to determine $\gamma s$ for $s \in S_{k_0+k}$ look at the coset $\Delta_{k}$. All permutations in this coset map $s$ to the same vertex $t$. Choose $\gamma s = t$. Since the sets $\Delta_k$ are nested $\gamma$ is well defined. 

It follows easily from subdegree finiteness that $\gamma$ is bijective and hence a permutation. Simply observe that if $\gamma_i$ and $\gamma_j$ are in $\Delta _k$ then $\gamma _i s = \gamma _j s$ and hence $\gamma _i^{-1} \gamma _j s = s$ for every $s \in S_0$. By subdegree finiteness there are only finitely many possible values for $\gamma _i^{-1} \gamma _j t$ for every $t \in S$ and hence there are only finitely many values for $\gamma _i^{-1} u$ (recall that $\gamma_j$ is bijective) for every $u \in S$. Now choose $k$ such that all of the possible values are contained in $S_{k_0+k}$. Then all permutations in $\Delta_k$ will map the same vertex to $u$ and hence $u$ has a preimage under $\gamma$.

If we can find a subsequence of $\gamma_i$ which converges to $\gamma$ in the set $\Pi_S$ of all permutations of $S$, then it follows that $\gamma \in \Gamma$ since $\Gamma$ is closed in $Pi_S$. Furthermore in this case we found a convergent subsequence of $\gamma_i$, which completes the proof of the lemma.

To construct such a subsequence choose $i_k$ such that $i_k>i_{k-1}$ and $\gamma_{i_k} \in \Delta_k$. Since $\gamma$ coincides with $\gamma_{i_k}$ on $S_{k+i_0}$ it follows that $\dist(\gamma_{i_k}, \gamma) \to 0$ as $k \to \infty$, so $\gamma_{i_k}$ converges to $\gamma$. 
\end{proof}

We conclude this section with a list of topological properties of $\Gamma$ which follow from the above results by well known theorems from topology. Let $\Gamma$ be a closed, subdegree finite group of permutations of a set $S$. Then each of the following holds:
\begin{itemize}
\item $\Gamma$ is $\sigma$-compact because there are only countably many distinct balls of radius $r<1$ and those balls are compact,
\item Lindelöf because every $\sigma$-compact space is Lindelöf,
\item separable and second countable because in a metric space these properties are equivalent to the Lindelöf property,
\item totally disconnected because in an ultrametric space balls are both open and closed,
\item locally compact because small balls are compact,
\item complete and hence Polish because small balls are compact and every Cauchy sequence will eventually stay within a small ball.
\end{itemize}

\section{Properties of stabiliser subgroups}
\label{sec:stabiliser}

In this section we outline some basic properties of stabiliser subgroups of colourings, partial colourings, and subsets of $S$. We start with a well known result about the stabiliser of a single element $s$ of $S$. 

\begin{lem}
\label{lem:stab_v}
Let $\Gamma$ be a closed, subdegree finite group of permutations of a countable set $S$. Then for every $s \in S$ the stabiliser $\Gamma_s$ is a compact subgroup of $\Gamma$.
\end{lem}

\begin{proof}
It is clear that the stabiliser must be a subgroup of $\Gamma$ so we only need to show that it is compact.
In the construction of the metric choose $S_1 = \{s\}$. Then $\Gamma_s$ is the ball centered at $\id$ with radius $\varrho = \frac 1 2 $. Hence it is compact by Lemma~\ref{lem:compactballs}.
\end{proof}

A similar result can also be obtained for the setwise stabiliser of a finite subset $S' \subseteq S$. In fact, the following lemma exactly tells us, when a closed and subdegree finite group of permutations of a countable set is compact.

\begin{lem}
\label{lem:whencompact}
Let $\Gamma$ be a closed, subdegree finite group of permutations of a countable set $S$. Then the following are equivalent:
\begin{enumerate}
\item \label{itm:whencompact:compact} $\Gamma$ is compact.
\item \label{itm:whencompact:setwise} $\Gamma$ stabilises some finite subset $S'$ of $S$ setwise.
\item \label{itm:whencompact:someorbit} The orbit of some element $s \in S$ is finite.
\item \label{itm:whencompact:allorbits} All orbits under the action of $\Gamma$ are finite.
\end{enumerate}
\end{lem}

\begin{proof}
Clearly \ref{itm:whencompact:allorbits} $\Rightarrow$ \ref{itm:whencompact:someorbit}. The implication \ref{itm:whencompact:someorbit} $\Rightarrow$ \ref{itm:whencompact:setwise} follows from the fact that $\Gamma$ stabilises every orbit setwise. The converse implication follows from the fact that the orbit of $s \in S'$ must be contained in $S'$ if the set is setwise stabilised. So we only need to show the implications \ref{itm:whencompact:someorbit} $\Rightarrow$ \ref{itm:whencompact:compact} $\Rightarrow$ \ref{itm:whencompact:allorbits} in order to prove the equivalence of the statements.

First assume that there is some $s \in S$ such that the orbit $\Gamma s$ is finite. Clearly $\Gamma$ is the union of the (finitely many) cosets with respect to the stabiliser $\Gamma_s$. All of the cosets are compact because the stabiliser is compact by Lemma~\ref{lem:stab_v}. Hence we have decomposed $\Gamma$ into finitely many compact sets and thus $\Gamma$ itself must be compact.

Conversely, let $\Gamma$ be compact and assume that there is some $s \in S$ whose orbit is infinite. Then we can find an infinite sequence $(\gamma_i)_{i \in \mathbb N}$ of permutations in $\Gamma$ such that no two permutations map $s$ to the same point. Since $\Gamma$ is compact this sequence must have a convergent subsequence which is impossible because no two permutations coincide on $s$ which gives a lower bound on their distance.
\end{proof}

Next we would like to turn to stabilisers of colourings of $S$. In general such a stabiliser will not be compact, but we can show that it is always a closed subgroup of $\Gamma$.

\begin{lem}
\label{lem:stab_c}
Let $\Gamma$ be a group of permutations of a countable set $S$. Then the stabiliser $\Gamma_c$ of a colouring $c$ of $S$ is a closed  subgroup of $\Gamma$.
\end{lem}

\begin{proof}
Again it is clear that the stabiliser of $c$ is a subgroup of $\gamma$ since $c \gamma = c \circ \gamma$ defines a right action of $\Gamma$ on the set $\mathcal C (S)$ of colourings of $S$. Hence we only need to show that it is closed.

Consider a permutation $\gamma \notin \Gamma_c$. There must be some $s \in S$ such that $c(s) \neq c(\gamma s)$. This point is contained in some set $S_i$, where $(S_i)_{i \in \mathbb N}$ is the nondecreasing sequence of finite subsets of $S$, which was used to construct the metric in Section~\ref{sec:metric}. Now, every permutation $\gamma'$ with $\dist(\gamma', \gamma)<2^{-i}$ coincides with $\gamma$ on $S_i$. This implies that no permutation in the ball $B_{\gamma} (2^{-i})$ is contained in $\Gamma_c$. So $\gamma$ has an open neighbourhood  which is disjoint to $\Gamma_c$ and hence the complement of $\Gamma_c$ is open.
\end{proof}

What happens if we consider partial colourings instead of colourings? It is readily verified that the stabiliser of a partial colouring $c'$ is in general not a subgroup of $\Gamma$, so we cannot hope for a verbatim extension of Lemma~\ref{lem:stab_c} to partial colourings. But it turns out that apart from the group property everything generalises nicely. If the domain of the partial colouring is finite we even get a better result: in this case the stabiliser will be a set that is both closed and open in the permutation topology.

\begin{lem}
\label{lem:stab_pc}
Let $\Gamma$ be a group of permutations of a countable set $S$ and let $c'$ be a partial colouring of $S$. Then the stabiliser of $c'$ is closed. If the domain of $c'$ is finite then the stabiliser is also open.
\end{lem}

\begin{proof}
Denote by $S'$ the domain of $c'$. Clearly, a permutation $\gamma \in \Gamma$ preserves $c'$ if and only if there is a colouring $c''$ of the set  
\[
T = S' \cup \gamma^{-1}S'
\]
such that for every $s \in S'$ it holds that $c'' (\gamma s) = c''(s) = c'(s)$.

If $S'$ is finite then so is $T$ and hence $T$ is contained in $S_i$ for some $i \in \mathbb N$. Consider a permutation $\gamma'$ such that $\dist (\gamma, \gamma') < 2^{-i}$. It follows from the definition of $\dist$ that $\gamma' s = \gamma s$ for every $s \in T$. Hence a colouring of $T$ with the above property exists for $\gamma$ if and only if it exists for $\gamma'$. It follows that if $\gamma \in \Gamma_{c'}$ then the ball with center $\gamma$ and radius $2^{-i}$ is completely contained in the stabiliser of $c'$ showing that the stabiliser is open. Conversely, if $\gamma \notin \Gamma_{c'}$ then this ball will be completely contained in the complement of the stabiliser, proving that the complement is open as well.

Now let us turn to the case where $S'$ is infinite. In this case choose a sequence $S_i'$ of finite subsets of $S'$ such that $S_i' \subseteq S_{i+1}'$ and $\lim_{i \to \infty} S_i' = S'$. Let $c_i'$ be the colouring with domain $S_i'$ which coincides with $c'$ on $S_i'$. We know that $\Gamma_{c_i'}$ is closed because of the first part of the proof. If we can show that $\Gamma_{c'} = \bigcap _{i \in \mathbb N} \Gamma_{c_i'}$ then it is closed because it is the intersection of closed sets.

But this is easy: if a permutation is contained in $\Gamma_{c'}$ then it is clearly contained in every $\Gamma_{c_i'}$ (simply use the same colourings to extend $c'$ and $c_i'$). If a permutation $\gamma$ is not contained in $\Gamma_{c'}$ then this means that there is no partial colouring with domain $T$ such that $c''(\gamma s) = c''(s)$ for each $s \in S'$. since we can colour every $s \in T \setminus S'$ arbitrarily this implies that there are two elements $s,t \in S'$ with different colours such that $\gamma s = t$. now choose $i$ large enough that $s,t \in S_i'$. Clearly $\gamma \notin \Gamma_{c_i'}$ and hence $\gamma$ is not contained in the intersection.
\end{proof}

\section{Random colourings}
\label{sec:randomcolour}

In this section we investigate properties of random colourings with respect to permutation breaking. By a random colouring we mean a $2$-colouring of the set $S$ where the colour of every element $s \in S$ is chosen independently and uniformly. The probability space that we obtain this way is $\{0,1\}^{\vert V \vert}$ with the product probability measure denoted by $\prob$.

The motivation to use random colourings comes from the following lemma due to Russel and Sundaram \cite{MR1617449}, or more precisely from its proof, which uses random colourings to obtain a distinguishing colouring for a finite graph with large motion.

\begin{lem}
\label{lem:motion}
Let $G$ be a graph with motion $m$ and assume that $2^{\frac m 2} \geq \vert \Aut G\vert $. Then there is a distinguishing $2$-colouring of $G$.
\end{lem}

\begin{proof}
As we mentioned before this fact can be shown using random colourings. Let $\varphi \in \Aut G \setminus \{ \id \}$. We know that $\varphi$ moves at least $m$ vertices, which implies that there are at most $\frac m 2$ cycles of length $\geq 2$ and $(n - m)$ singleton cycles in the corresponding permutation. For a random colouring $c$ we have
\begin{align*}
\prob [c \varphi = c] 
&= \prob [\text{all cycles are monochromatic}]\\
&\leq  \frac 1 {2^n} \left( 2^{n-\frac{m}{2}} \right)\\
&= 2^{-\frac m 2}
\end{align*}
Summing up those estimates for $\varphi \in \Aut G \setminus \{\id \}$ we get that

\begin{align*}
\prob [\exists \varphi \colon c \varphi = c] 
&\leq \sum_{\varphi \in \Aut G \setminus \{\id \}} \prob [ c \varphi = c]\\
&\leq \sum_{\varphi \in \Aut G \setminus \{\id \}} 2^{-\frac m 2}\\
&\leq \left( 2^{\frac m 2}-1\right) 2^{-\frac m 2}\\
&<1.
\end{align*}
So a random colouring has a positive probability of being distinguishing and hence there must be such a colouring.
\end{proof}

Notice that the proof does not use the graph structure or group structure in any way. Hence we can apply the same arguments to prove the following (stronger) statement.

\begin{lem}
Let $S$ be a set and let $\Delta$ be a set of permutations of $S$ with motion $m$. Assume that $2^{\frac m 2} < \vert \Delta \vert$. Then there is a distinguishing $2$-colouring. {\hfill \qedsymbol}
\end{lem}

For locally finite graphs with infinite motion the inequality in the statement of the previous lemma (seen as an inequality of infinite cardinals) is always fulfilled. So Conjecture~\ref{con:tucker} can be seen as an infinite analogue to Lemma~\ref{lem:motion}. Hence it is only natural to ask whether a similar proof can be used in the locally finite case. If the inequality is strict then the following theorem of Halin \cite{MR0335368}, which is independent from the continuum hypothesis tells us that the automorphism group must be countable.

\begin{thm}
\label{thm:halin}
Let $G$ be a locally finite graph then $\vert \Aut G \vert < 2^{\vert \mathbb N \vert}$ if and only if there is a finite subset of $V$ whose pointwise stabiliser is trivial. {\hfill \qedsymbol}
\end{thm}

Clearly, if there is such a set then an automorphism is uniquely determined by the image of this set. Since there are only countably many possibilities to map a finite set to a countable set the automorphism group can be at most countable.

A similar thing holds---again independently from the continuum hypothesis---for closed permutation groups by the following result of Evans \cite{0635.20001}.

\begin{thm}
If $\Gamma$ and $\Delta$ are closed permutation groups on a countable set $S$ and $\Delta \subseteq \Gamma$, then either $\vert \Gamma : \Delta \vert = 2^{\aleph_0}$ or $\Delta$ contains the pointwise stabiliser of some finite set in $\Gamma$. {\hfill \qedsymbol}
\end{thm}

Taking $\Delta = \{ \id \}$ in the above theorem we obtain that a closed permutation group $\Gamma$ either has cardinality $2^{\aleph _0}$, or there is some finite subset of $S$ whose pointwise stabiliser is trivial. In particular, Theorem~\ref{thm:halin} and all of its implications remain true in the more general setting of closed permutation groups.

It is known that a countable group of permutations with infinite motion of a countable set admits a distinguishing $2$-colouring \cite{istw}. The following theorem shows that almost every 2-colouring has this property, its proof is basically a copy of the proof of Lemma~\ref{lem:motion}.

\begin{thm}
\label{thm:countable}
Let $\Gamma$ be a countable group of permutations with infinite motion of a countable set $S$ and let $c$ be a random colouring of $S$. Then $c$ is almost surely $\Gamma$-distinguishing.
\end{thm}

\begin{proof}
Clearly, for any given permutation $\gamma \in  \Gamma$ it follows from infinite motion that there are infinitely many disjoint pairs $(s_i, \gamma s_i) \in S\times S$. If we would like $\gamma$ to preserve the colouring it is necessary that all of those pairs are monochromatic. However, for each pair this only happens with probability $\frac 12$. So there is almost surely a pair $(s_i, \gamma s_i)$ carrying different colours and hence $\gamma$ is almost surely broken by $c$.

In order to see that $c$ is almost surely distinguishing we use $\sigma$-subadditivity of the probability measure $\prob$:
\[
 \prob [\exists \gamma \colon c\gamma=c] \leq \sum_{\gamma \in \Gamma} \prob [c \gamma =c] = 0
\]
because every summand is 0.
\end{proof}

Just like Lemma~\ref{lem:motion} the proof of the above result is easily seen to be independent of the group structure of $\Gamma$. 

The argument fails when $\Gamma$ is uncountable because summation is no longer possible. However, we know that a closed, subdegree finite group of permutations is always separable. Applying our argument to a dense countable subset yields the following.

\begin{thm}
Let $\Gamma$ be a separable group of permutations of a countable set $S$ with infinite motion and let $c$ be a random colouring of $S$. Then $\Gamma_c$ is almost surely nowhere dense in $\Gamma$.
\end{thm}

\begin{proof}
Choose a dense countable subset of $\Gamma$. By the same arguments as before the random colouring $c$ almost surely breaks every automorphism in this subset. By Lemma~\ref{lem:stab_c} the stabiliser of $c$ is a closed subgroup, hence its complement is almost surely an open dense set. This implies that the stabiliser must be almost surely nowhere dense in $\Gamma$.
\end{proof}

In \cite{istw} it is shown that every closed permutation group has a dense subgroup for which there is a distinguishing 2-colouring. Observing that the subgroup generated by a countable set is again countable we get the same result for every separable permutation group.

So far we have shown that, if $\Gamma$ is closed and subdegree finite, then  the stabiliser subgroup of a random colouring is almost surely topologically sparse, which was more or less a direct consequence of separability. But it turns out that under suitable conditions the set of unbroken permutations is small in at least one more way: it almost surely has Haar measure 0. The basic proof ideas come again from the proof of Theorem~\ref{thm:countable}, the main difference being that we replace the sum by an integral with respect to the Haar measure. In the proof we will need the following version of Fubini's theorem which can, for example, be found in \cite{MR924157}.

\begin{thm}
\label{thm:fubini}
Let $(X, \mathcal X, \mu)$ and $(Y, \mathcal Y, \nu)$ be $\sigma$-finite measure spaces and let $f\colon X \times Y \to \mathbb R$ be a  non-negative, $(\mathcal X \times \mathcal Y)$-measurable function. Then
\[
\int _X \left( \int _Y f(x,y) \,d\nu(y) \right)\,d\mu(x) = \int _Y \left( \int _X f(x,y) \,d\mu(x) \right) \,d\nu(y). \tag*{\qedsymbol}
\] 
\end{thm}

Having stated this theorem we are now ready to prove the following.

\begin{thm}
\label{thm:zeromeasuregrp}
Let $\Gamma$ be a closed, subdegree finite group of permutations of a countable set $S$ and assume that the motion of $\Gamma$ is infinite. Then a random colouring $c$ almost surely breaks almost every (with respect to the Haar measure) element of $\Gamma$.
\end{thm}

\begin{proof}
First of all recall that $\Gamma$ is locally compact by arguments in Section~\ref{sec:metric}. So we can define a Haar measure $\haar$ on $\Gamma$.

We now claim that  for a random colouring $c$ the expected value of $\haar (\Gamma _c)$ is 0. Since $\haar(\Gamma_c)$ is a non-negative random variable this implies that $\haar(\Gamma_c)=0$ almost surely which proves the lemma.

To see that the expected value is indeed 0 we calculate
\begin{align*}
\mathbb E(\haar(\Gamma_c)) 
&= \int_{\mathcal C(S)} \haar(\Gamma_c) \, d\prob (c) \\
&= \int_{\mathcal C(S)} \int_{\Gamma} I_{[c\gamma = c]} \,d\haar (\gamma) \,d\prob(c)
\end{align*}

Since $\Gamma$ is the union of countably many compact balls by Lemma~\ref{lem:fewballs} and Lemma~\ref{lem:compactballs}, it is a $\sigma$-compact space. Compact sets have finite Haar measure, hence the Haar measure on $\Gamma$ is $\sigma$-finite.

In order to be able to apply Theorem~\ref{thm:fubini}, we still need to show that the function which we would like to integrate is measurable. Since it is the indicator function of the set
\[
	U = \{(c,\gamma) \in \mathcal C(S) \times \Gamma \mid c \gamma = c \}
\]
it suffices to show that $U$ is measurable. For this purpose let $(S_i)_{i \in \mathbb N}$ be a sequence of finite subsets of $S$ such that $\lim _{i \to \infty} S_i = S$. For each partial colouring $c'$ with domain $S_i$ define
\[
	U_i (c') =  \{(c,\gamma) \in \mathcal C(S) \times \Gamma_{c'} \mid \forall s \in S_i\colon c(s) = c'(s)\}.
\]
Observe that $C = \{c \in \mathcal C(S) \mid \forall s \in S_i\colon c(s) = c'(s)\}$ is a cylinder set and $\Gamma_{c'}$ is open and closed by Lemma~\ref{lem:stab_pc}. Since $U_i(c') = C \times \Gamma_{c'}$ it is clearly contained in the product $\sigma$-algebra.

Now let
\[
	U_i = \bigcup_{c' \in \mathcal C(S_i)} U_i (c').
\]
This set is measurable because it is the finite union of measurable sets. We claim that
\[
	U = \bigcap _{i \in \mathbb N} U_i.
\]
To see that this is indeed the case consider $(c,\gamma) \in U$. Clearly $c$ coincides with some partial colouring $c'$ on $S_i$ and $\gamma$ preserves this partial colouring because it preserves $c$. Hence $(c,\gamma)$ is contained in every $U_i$ and thus also in the intersection. 

Conversely, let $(c, \gamma) \in \bigcap _{i \in \mathbb N} U_i$. Assume that $\gamma$ does not preserve $c$. Then there is $s \in S$ such that $c(s) \neq c(\gamma s)$. Take $i$ large enough that $s$ and $\gamma s$ are contained in $S_i$. Clearly, $\gamma$ does not preserve the partial colouring $c'$ which coincides with $c$ on $S_i$. Hence $(c,\gamma) \notin U_i$, a contradiction to $(c, \gamma) \in \bigcap _{i \in \mathbb N} U_i$. 

Altogether we have shown that $U$ can be written as a countable intersection of measurable sets. So it is measurable itself and hence the indicator function $I_U =  I_{[c\gamma=c]}$ is measurable as well.

This implies that we can apply Fubini's theorem to the iterated integral above and obtain 
\begin{align*}
\mathbb E(\haar(\Gamma_c)) 
&= \int_{\Gamma} \int_{\mathcal C(S)}  I_{[c \gamma = c]}  \,d\prob(c) \,d\haar(\gamma) \\
&= \int_{\Gamma} \prob[c \gamma = c] \,d\haar(\gamma).
\end{align*}

We already observed earlier that the probability that a given permutation preserves a random colouring is $0$, hence we integrate over the constant $0$-function and thus the integral evaluates to $0$.
\end{proof}

We conclude this section with a result which will be useful in the next section. In order to state this result, we need the following equivalence relation which is closely linked to the distinct spheres condition which was introduced in \cite{smtuwa} as a sufficient condition for $2$-distinguishability. We will see later (Theorem~\ref{thm:distinct_spheres}), that for locally finite graphs satisfying this condition a random $2$-colouring is almost surely distinguishing. 

Let $S$ be a countable set and let $\Gamma$ be a subdegree finite group of permutations of $S$ with infinite motion. Define an equivalence relation $\sim_\Gamma$ on the set $S$ as follows: two points $s,t \in S$ are called $\Gamma$-equivalent, if the following holds:
\begin{itemize}
\item there is a permutation $\varphi \in \Gamma$ such that $\varphi s = t$ and 
\item for all but finitely many $x \in S$ the orbits $\Gamma _s x$ and $ \varphi \Gamma_s x $ coincide.
\end{itemize}
Notice that the latter requirement is true for $\varphi$ if and only if it is true for every $\gamma$ such that $\gamma s = t$ because in this case $\gamma = \varphi \gamma_s$ for a suitable $\gamma_s \in \Gamma_s$. Hence, if the second condition does not depend on the choice of $\varphi$.

\begin{prp} 
The relation $\sim_\Gamma$ is indeed an equivalence relation.
\end{prp}

\begin{proof}
To show reflexivity simply choose $\varphi = \id$.

For symmetry assume that $s \sim_\Gamma t$ and let $\varphi \in \Gamma$ such that $\varphi s = t$. Notice that $\Gamma_s = \varphi^{-1} \Gamma_t \varphi$, so for $\Gamma_s x = \varphi \Gamma_s x$ we have
\[
	\Gamma_t \varphi x = \varphi \Gamma_s x = \Gamma_s x = \varphi^{-1} \Gamma_t \varphi x.
\]
This implies that $\Gamma_t y = \varphi^{-1} \Gamma_t y$ for all but finitely many values of $y = \varphi x$, that is $t \sim _\Gamma s$.

Finally, we need to show transitivity. Assume that $s \sim_\Gamma t$ and that $t \sim_\Gamma u$ and let $\varphi$ and $\psi$ be the corresponding permutations. By definition this implies that for all but finitely many $x \in S$ it holds that $\varphi \Gamma_s x = \Gamma_s x$ and $\psi \Gamma _t \varphi x = \Gamma _t \varphi x$. Using the fact that $\varphi \Gamma _s = \Gamma_t \varphi$ we obtain
\[
	\psi \varphi \Gamma _s x =\psi \Gamma _t \varphi x = \Gamma _t \varphi x = \varphi \Gamma _s x = \Gamma _s x
\]
for all but finitely many $x \in S$.
\end{proof}

We denote the equivalence class of $s \in S$ with respect to $\sim_\Gamma$ by $[s]_\Gamma$. With the above notation we have the following lemma, which constitutes a generalisation of a result in \cite{smtuwa} as we will see in section~\ref{sec:graphcolour}.

\begin{lem}
\label{lem:fixclasses}
Let $S$, $\Gamma$, and $\sim_\Gamma$ be defined as above. Assume that $\Gamma$ has infinite motion and let $c$ be a random colouring of $S$. Then $c$ almost surely fixes every equivalence class with respect to $\sim_\Gamma$, that is
\[
\Gamma_c \subseteq \bigcap_{s \in S} \Gamma_{[s]_\Gamma}
\]
 where $\Gamma_{[s]_\Gamma}$ denotes the setwise stabiliser of $[s]_\Gamma$.
\end{lem}

\begin{proof}
For $t \nsim _\Gamma s$ and $u \in S$ consider the event 
\[
A_{stu} = [\exists \varphi \in \Gamma_c \colon \varphi s =t, \varphi t = u].
\]
If we can show that the probability of $A_{stu}$ is $0$ we are done, because
\[
	\prob[\exists s \nsim _\Gamma t, \varphi \in \Gamma \colon \varphi s = t] 
	= \prob \left( \bigcup _{s \in S} \bigcup _{\substack{t \in S \\ s \nsim_{\Gamma} t}} \bigcup_{u \in S} A_{stu}\right) 
	\leq \sum_{s \in S} \sum_{\substack{t \in S\\t\nsim_\Gamma s}} \sum_{u \in S} \prob (A_{stu}) =0.
\]

So let us take a closer look at $\prob (A_{stu})$. If there is no permutation in $\Gamma$ which maps $s$ to $t$ and $t$ to $u$, then this probability clearly is $0$. So assume that there is such a permutation $\varphi$. Let $v \in V$. Since $s$ is mapped to $t$ the set $\Gamma _s v$ must be mapped to the set $\varphi \Gamma _s v$. Notice that the set $\varphi \Gamma _s v$ does not depend on the particular choice of $\varphi$, that is, it is the same for every $\varphi \in \Gamma$ with $\varphi s = t$. In particular this implies that if the set $ \Gamma _s v \setminus \varphi \Gamma _s v$ is nonempty, then it will be mapped to the disjoint set $\varphi \Gamma _s v \setminus \varphi ^2 \Gamma _s v$ by every automorphism which maps $s$ to $t$. The set $\varphi ^2 \Gamma _s v$ depends only on $u$, that is, the image of $s$ under $\varphi^2$ and not on the particular choice of $\varphi$.

There are infinitely many points $v$ for which these difference sets are nonempty because $s \nsim_\Gamma t$ and each of the sets is finite because of subdegree finiteness. Hence we can choose infinite sequences of non empty, disjoint sets $P_i:=\Gamma _s v_i \setminus \varphi \Gamma _s v_i$ and $Q_i = \varphi \Gamma _s v_i \setminus \varphi^2 \Gamma _s v_i$ such that all of the $P_i$ and $Q_j$ are also pairwise disjoint for all $i,j \in \mathbb N$. 

Now assume that there is a colour preserving permutation which maps $s$ to $t$. This can only happen if the sets $P_i$ and $Q_i$ contain the same number of vertices of each colour for every $i$. Let $n_i := \vert P_i \vert = \vert Q_i \vert$ and denote by $p_i$ and $q_i$ the number of elements of $P_i$ and $Q_i$ with colour $0$ respectively. Then the probability that the colour distributions on $P_i$ and $Q_i$ coincide can be expressed as
\begin{align*}
\prob[p_i = q_i] 
&=  \sum_{j=1}^{n_i} \prob [p_i = j\mid q_i =j]\, \prob [q_i = j]\\ 
&=  \sum_{j=1}^{n_i} \prob [p_i = j]\, \prob [q_i = j]\\ 
&= \sum_{j=1}^{n_i}\, \binom{n_i}{j} 2^{-n_i} \binom{n_i}{j} 2^{-n_i} 
\end{align*}
where the second equality follows from the fact that $P_i$ and $Q_i$ are disjoint and hence their colourings are independent. To get an estimate for the last sum observe that 
\[
	\binom{n_i}{j} 2^{-n_i} \leq \frac 1 2,
\]
and hence 
\[
\prob [p_i = q_i] \leq \frac 1 2 \sum_{j=1}^{n_i}\, \binom{n_i}{j} 2^{-n_i} = \frac 1 2.
\]

Recall that in order to have an automorphism which maps $s$ to $t$ we need $p_i = q_i$ for every $i \in \mathbb N$. These events will be independent because all of the sets are disjoint. Hence we have
\[
	\prob[\exists \varphi \in \Gamma_c \mid \varphi s = t] \leq \prod_{i \in \mathbb N} \prob[p_i = q_i] = 0. \qedhere
\]
\end{proof}

\begin{rmk}
Notice that Lemma~\ref{lem:fixclasses} can be iterated as follows. Let $\Gamma ^ 0 = \Gamma$ and denote by $\sim_0$ the relation $\sim_{\Gamma ^ 0}$. Inductively, for $i \geq 0$ define 
\[
	\Gamma ^ {i+1}= \bigcap_{v \in V} \Gamma ^ i _{[v]_i},
\]
where $[v]_i$ is the equivalence class of $v$ with respect to $\sim_i$ and $\Gamma ^ i _{[v]_i}$ is its setwise stabiliser. Define $\sim_{i+1} = \sim_{\Gamma^{i+1}}$.

Let $c$ be a random colouring of $S$. Inductively applying Lemma~\ref{lem:fixclasses} we obtain that almost surely $\Gamma_c \subseteq \Gamma^i$ for each $i \in \mathbb N_0$. This implies that almost surely
\[
	\Gamma_c \subseteq \Gamma ^\infty = \lim_{i \to \infty} \Gamma^i = \bigcap _{i \in \mathbb N_0} \Gamma^i.
\]
\end{rmk}

\begin{rmk}
The set of permutations in $\Gamma$, that fix all equivalence classes with respect to $\sim_\Gamma$ setwise is a group. Denote it by $\Delta$. If there is a finite equivalence class then  Lemma~\ref{lem:whencompact} implies that $\Delta$ is compact and hence the stabiliser of a random colouring is almost surely compact.

But even if it is not compact $\Delta$ is the limit of a sequence of compact subgroups. Simply notice that for a fixed $s \in S$ every permutation $\varphi \in \Delta$ must fix all but finitely many suborbits $\Delta_s x$ setwise. Let $\Delta_s x_i$ be an enumeration of all suborbits and define
\[
	\Delta_i = \{ \varphi \in \Delta \mid \forall j > i \colon \varphi \Delta _s x_j = \Delta _s x_j \}.
\]
Then $\Delta_i$ is compact by Lemma~\ref{lem:whencompact} because the $\Delta_i$-orbit of $x_j$ is contained in the finite suborbit $\Delta _s x_j$ for $j > i$. Clearly the sequence $\Delta_i$ is nondecreasing and every $\varphi \in \Delta$ is contained in some $\Delta_i$. Thus
\[
	\Delta = \lim_{i \to \infty} \Delta_i = \bigcup _{i \in \mathbb N_0} \Delta_i.
\]
\end{rmk}

The above remark also tells us, that in order to prove Conjecture~\ref{con:random} it suffices to consider compact groups. More precisely we have the following.

\begin{cor}
Assume that for every compact, subdegree finite permutation group with infinite motion a random colouring is almost surely distinguishing. Then the same is true for every subdegree finite permutation group with infinite motion.
\end{cor}

\begin{proof}
With the above notation every nontrivial permutation will be contained in some $\Delta_i$. Since $\Delta_i$ is compact the stabiliser in $\Delta_i$ of a random colouring $c$ will almost surely be trivial. By $\sigma$-subadditivity of the probability measure we get
\[
\prob [\Gamma_c \text{ is not trivial}] \leq \sum_{i=1}^\infty \prob[(\Delta_i)_c \text{ is not trivial}] = 0. \qedhere
\]
\end{proof}

\section{Random graph colourings}
\label{sec:graphcolour}

The last section of this paper is devoted to random colourings of graphs. First of all notice, that the automorphism group of a locally finite graph is always a closed, subdegree finite group of permutations on the vertex set. Hence all results from the previous section apply to automorphism groups of locally finite graphs as well.

\begin{thm}
Let $G$ be a locally finite graph with infinite motion and let $c$ be a random colouring of $G$. Then $(\Aut G)_c$ is almost surely a nowhere dense, closed subgroup with Haar measure $0$. {\hfill \qedsymbol}
\end{thm}
   
If, instead of colouring all vertices randomly, we first colour part of the vertices deterministically we can even make the stabiliser subgroup of the resulting colouring compact as the following theorem shows.
   
\begin{thm}
\label{thm:compactgrp}
Let $G$ be a locally finite graph with infinite motion. Then there is a colouring of $G$ which is only stabilised by a nowhere dense, compact subgroup with measure $0$ of $\Aut G$.
\end{thm}

In order to prove this theorem we first need the following auxiliary result from \cite{cuimle}:

\begin{lem}
\label{lem:fixroot}
Let $G =(V,E)$ be an infinite, locally finite, connected graph with infinite motion, $v_0 \in V$. For every $\delta >0$ there is a partial colouring $c'$ of the vertices of $G$ with the following properties:
\begin{enumerate}
\item $c'$ is $A$-distinguishing for $A= \{\varphi \in \Aut G  \mid \varphi v_0 \neq v_0 \}$.
\item There is $k_0$ such that less than $\delta k$ of the spheres $\sph{m+1}, \ldots , \sph {m+k}$ are coloured for every $k > k_0$ and every $m \in \mathbb N$.{\hfill \qedsymbol}
\end{enumerate}
\end{lem}

\begin{proof}[Proof of Theorem~\ref{thm:compactgrp}]
First apply Lemma~\ref{lem:fixroot} in order to break all automorphisms which move a given vertex $v_0$. This gives a partial colouring $c'$ of the graph which is by Lemma~\ref{lem:stab_pc} only preserved by a closed subset of $\Aut G$. $\Gamma_{c'}$ is compact because it is completely contained in $(\Aut G)_{v_0}$ which is compact by Lemma~\ref{lem:stab_v}. 

It is easy to see that $\Gamma_{c'}$ has infinite motion on the set of yet uncoloured vertices. Now let $c$ be the coloring obtained by randomly colouring all vertices that have not been coloured yet. We can apply Theorem~\ref{thm:zeromeasuregrp} to show that $c$ almost surely breaks almost every remaining automorphism of $G$.

$\Gamma_c$ forms a closed and hence compact subgroup of $(\Aut G)_{v_0}$. Since it has measure $0$ in $(\Aut G) _{v_0}$ it must also have measure $0$ in $\Aut G$. The property of being nowhere dense also carries over from $(\Aut G) _{v_0}$ to $\Aut G$.
\end{proof}

In the rest of this paper we will show that there are many classes of locally finite graphs for which Conjecture~\ref{con:random} can be verified. Before doing so, however, we would like to point out that the requirement of local finiteness is necessary. Otherwise the conjecture fails even for trees as the following theorem shows.

\begin{thm}
Denote by $T_\infty$ the regular tree with countably infinite degree and let $c$ be a random colouring of $T_\infty$ with finitely many colours. Then there is almost surely an automorphism of $T_\infty$ which preserves $c$.
\end{thm}

\begin{proof}
First of all notice that in a random colouring every vertex has infinitely many neighbours of each colour. Hence it suffices to find a nontrivial automorphism preserving a coloring with this property.

Let $c$ be such a coloring and choose a vertex $v_0$ of $T_\infty$. Define $\varphi v_0 = v_0$. Next choose an arbitrary colour-preserving permutation $\pi$ of the neighbours of $v_0$ and define $\varphi v = \pi v$ for every neighbour $v$ of $v_0$. 

Now assume that $\varphi$ has already been defined for all vertices $v$ with $d(v, v_0) \leq n$. For a vertex $v$ with $d(v, v_0) = n$ let $(w_i^{(v,j)})_{i \in \mathbb N}$ be an enumeration of the neighbours of $v$ with colour $j$ which lie further away from $v_0$ then $v$. Recall that there are always countably many such neighbours, hence the sequence will be infinite.

Now define $\varphi w_i^{(v,j)} = w_i^{(\varphi v,j)}$. Clearly this assignment is bijective if the assignment on $S_{v_0}(n)$ is bijective. Notice that this is the case since we started with a permutation for $n=1$. It is also straightforward to check that it preserves adjacency and colours.

Proceeding inductively we obtain the desired automorphism.
\end{proof}

In the remainder of this section we will focus on examples of graphs, where a random colouring is almost surely distinguishing. The following lemma which is a direct consequence of Lemma~\ref{lem:fixclasses} will be of great use.

Define the \emph{sphere around $v_0$ with radius $n$} by $S_{v_0}(n) = \{v \in V \mid d(v_0,v) = n\}$. We call two vertices $u$ and $v$ sphere equivalent ($u \sim_S v$), if there is an automorphism of $G$ which maps $u$ to $v$ and an integer $n_0 \in \mathbb N$ such that $S_u(n) = S_v(n)$ for every $n \geq n_0$. It is easy to verify that this is indeed an equivalence relation.

\begin{lem}
\label{lem:sim_s}
Let $G$ be a locally finite graph with infinite motion. A random colouring almost surely fixes all equivalence classes with respect to $\sim_S$ setwise.
\end{lem}

\begin{proof}
Recall that the automorphism group of a locally finite graph is always subdegree finite. For $\Gamma = \Aut G$ the relation $\sim_\Gamma$ defined in Section~\ref{sec:randomcolour} is finer than $\sim _S$. Since by Lemma~\ref{lem:fixclasses} a random colouring almost surely fixes every equivalence class with respect to $\sim_\Gamma$, it also almost surely fixes every equivalence class with respect to $\sim_S$.
\end{proof}

\subsection{The distinct spheres condition}

The \emph{distinct spheres condition (DSC)} was introtuced in \cite{smtuwa} as a sufficient condition for $2$-distinguishability of graphs. It is also shown that such a graph has infinite motion and hence supports Conjecture \ref{con:tucker}. In this subsection we show that if a locally finite graph satisfies DSC, then it is also supports Conjecture \ref{con:random}, that is, a random colouring is almost surely distinguishing.

First, let us define the condition. A graph $G=(V,E)$ is said to satisfy DSC if there is a vertex $v_0 \in V$ such that for any pair $x,y$ of distinct vertices $d(v_0,x) = d(v_0,y)$ implies that $S_x(n) \neq S_y(n)$ for infinitely many (or equivalently: all but finitely many) $n \in \mathbb N$. Clearly this implies that if $d(v_0,x) = d(v_0,y)$ then $x \nsim_S y$ and we can deduce the following result.

\begin{thm}
\label{thm:distinct_spheres}
If a locally finite graph $G = (V,E)$ satisfies DSC, then a random $2$-colouring $c$ is almost surely distinguishing.
\end{thm}

\begin{proof}
By Lemma~\ref{lem:sim_s} it suffices to show, that an automorphism which is contained in the setwise stabiliser of each equivalence class with respect to $\sim_S$ is necessarily the identity.

Let $\varphi$ be a non-trivial automorphism of $G$. If $\varphi v_0 \nsim_S v_0$ then $\varphi$ is not contained in the setwise stabiliser of all equivalence classes with respect to $\sim_S$.

So assume $\varphi v_0 \sim_S v_0$. If $\varphi v_0 = v_0$ then $\varphi$ stabilises all spheres with center $v_0$ setwise. Since $\varphi \neq \id$ there must be some $n \in \mathbb N$ such that $\varphi$ acts nontrivially on $\sph n$. 

If $\varphi v_0 \neq v_0$ but $\varphi v_0 \sim_S v_0$ then there is some $n_0 \in \mathbb N$ such that $\varphi$ stabilises $\sph n$ for $n>n_0$. Since $\varphi$ acts nontrivially on $B_{v_0}(n)$ it must also act nontrivially on the boundary $\sph n$.

Since $x \nsim_S y$ for any two vertices $x,y \in \sph n$ we can conclude that $\varphi$ again is not contained in the setwise stabiliser of all equivalence classes with respect to $\sim_S$.

Hence for every nontrivial automorphism of $G$ there is an equivalence class with respect to $\sim_S$ which is not setwise stabilised by $\varphi$.
\end{proof}

\begin{cor}
\label{cor:distinct_spheres}
Let $G$ be an infinite, locally finite graph. Then each of the following properties implies that a random $2$-colouring is almost surely distinguishing:
\begin{itemize}
\item $G$ is a leafless tree,
\item $G$ can be written as a product of two infinite factors,
\item the automorphism group of $G$ acts primitively on the vertex set,
\item $G$ is vertex-transitive and has connectivity $1$.
\end{itemize}
\end{cor}

\begin{proof}
All of these graphs satisfy DSC by \cite{smtuwa}.
\end{proof}

\subsection{Graphs with a global tree structure}

Trees are the probably most elementary example for a family of graphs which is known to satisfy Conjecture~\ref{con:tucker}. As we have seen, leafless trees also satisfy Conjecture~\ref{con:random}. The following corollary to Theorem~\ref{thm:distinct_spheres} shows, that the same holds true for arbitrary trees with infinite motion.

\begin{cor}
A random colouring of a locally finite tree with infinite motion is almost surely distinguishing.
\end{cor}

\begin{proof}
Since we assume infinite motion we can ignore finite subtrees and consider the subgraph induced by those vertices whose removal results in at least $2$ infinite components. On this set the relation $\sim_S$ is easily seen to be trivial. Alternatively one could note that the resulting graph is a leafless tree and hence stisfies the DSC.
\end{proof}

Tree like graphs are graphs with the following property: there is a vertex $v_0 \in V$ such that every vertex $v \in V$ has a neighbour $w$ such that $v$ lies on every shortest $w$-$v_0$-path. It is readily verified that this class of graphs again staisfies DSC.

\begin{cor}
A random colouring of a locally finite, tree like graph is almost surely distinguishing. {\hfill \qedsymbol}
\end{cor}

It is a well known fact that every graph has an end faithful spanning tree \cite{MR2159259}, that is, the ends of a graph can be seen as the ends of a spanning tree of the same graph. We now show that this large-scale tree structure is also almost surely preserved by every automorphism that preserves a random colouring. First of all we show that if $G$ has more than one end, then $\Gamma_c$ is almost surely compact and hence by Lemma~\ref{lem:whencompact} stabilises a finite set which plays the role of a root. Hence translations can only happen on a small scale. Then we show that such an automorphism almost surely fixes every end. Both of these results are again consequences of Lemma~\ref{lem:sim_s}.

\begin{lem}
\label{lem:twoends_comp}
Let $c$ be a random colouring of a locally finite graph with at least two ends. Then $(\Aut G)_c$ is almost surely compact.
\end{lem}

\begin{proof}
By Lemma~\ref{lem:whencompact} it suffices to show that there is a finite orbit which is the case if the equivalence class of some vertex $v$ with respect to $\sim_S$ is finite.

So let $v \in V$. There is a ball $B_v(n_0)$ such that $G \setminus B_v(n_0)$ has at least two infinite components. Assume that there is a vertex $w \sim_S v$ such that $d(v,w) \geq 2n_0+1$ and assume that $S_v(n) = S_w(n)$ for every $n>N$. 

Now notice that if $u$ lies in a different component of $G \setminus B_v(n)$ than $w$, then every path from $w$ to $u$ has to pass through $B_v(n_0)$. But this implies that $d(v,u) < d(w,u)$ since a shortest path from $v$ to $u$ takes $n_0$ steps before exiting $B_v(n_0)$ while a shortest $w$-$u$-path takes $n+1$ steps to reach $B_v(n_0)$.

So all vertices that are equivalent to $v$ must lie within the ball $B_v(2 n_0)$ which is finite.
\end{proof}

\begin{lem}
\label{lem:fixends}
Let $c$ be a random colouring of a locally finite graph. Then $(\Aut G)_c$ almost surely only contains automorphisms which fix the set of ends of $G$ pointwise.
\end{lem}

\begin{proof}
For one-ended graphs there is nothing to show, so we may assume that $G$ has at least $2$ ends. A random coloring is almost surely only preserved by automorphisms which stabilise the equivalence classes with respect to $\sim_S$ setwise. Hence it suffices to show that every such automorphism also fixes the set $\Omega$ of ends of $G$ pointwise. 

 Assume that $\varphi$ is contained in the setwise stabiliser of each equivalence class and that $\varphi \omega \neq \omega$ for some end $\omega$ of $G$. Let $(v_i)_{i \in \mathbb N}$ be a sequence of vertices converging to $\omega$. The sequence $(\varphi v_i)_{i \in \mathbb N}$ will converge to $\varphi \omega$ and hence $v_i$ and $\varphi v_i$ will lie in different infinite components of $G \setminus B_{v_0}(n)$ for large enough $n$ and $i$. By similar arguments as in the proof of the previous theorem this implies that $v_i \nsim_S \varphi v_i$ for large values of $i$.
 
 So $\varphi$ does not stabilise the equivalence classes with respect to $\sim_S$ setwise, a contradiction.
\end{proof}

\subsection{Cartesian products}

Another class of graphs where $2$-distinguishability results are known are Cartesian products. The \emph{Cartesian product} of two graphs $G_1 = (V_1,E_1)$ and $G_2 =(V_2,E_2)$ is the graph $G=(V,E)$ where $V = V_1 \times V_2$ and two vertices $(v_1,v_2)$ and $(w_1,w_2)$ are adjacent if $v_1w_1 \in E_1$ and $v_2 = w_2$ or $v_1 = w_1$ and $v_2w_2 \in E_2$. In this case we write $G = G_1 \sq G_2$. It is easy to see that the Cartesian product is associative and commutative, that is, the graphs obtained by changing the order in which Cartesian products are taken are isomorphic. We will use this fact throughout this section without explicitly mentioning it.

A \emph{$G_1$-layer} of $G = G_1 \sq G_2$ is the subgraph of $G$ induced by the set $\{ (v,v_2) \mid v \in V_1\}$ where $v_2 \in V_2$ is fixed. Analogously define a $G_2$-layer.

Throughout this section we will state state some well known facts about Cartesian products of graphs without proving them. All of the results and their proofs can be found in \cite{pre05906416}.

The first fact that we will need is, that the distance between two vertices in a Cartesian product is the sum of the distances of the projections to the factors. Hence a composition of shortest paths in the factors is a shortest path in the Cartesian product.

\begin{lem}
\label{lem:fixlayers}
Let $G$ be a locally finite graph with infinite motion which is not prime with respect to the Cartesian product. Choose a decomposition $G= G_1 \sq G_2$ such that $G_1$ is infinite. Let $c$ be a random colouring of $G$. Then $c$ almost surely fixes every $G_1$-layer setwise.
\end{lem}

\begin{proof}
Once again we would like to use Lemma~\ref{lem:sim_s}. So assume that there are two sphere equivalent vertices $v \sim_S w$ of $G$ which lie in different $G_1$-layers. 

Let $R = (v=v_0v_1v_2v_3\ldots)$ be a geodesic ray (that is, $d(v_0,v_i) = i$) starting in $v$ which remains inside the same $G_1$-layer forever. Denote by $R'= (v_0'v_1'v_2'v_3'\ldots)$ the ray in the layer of $w$ which is obtained from $R$ by only changing the $G_2$-coordinates.

Then $d(w,v_i) < d(w,v_i')$ while $d(v,v_i) > d(v,v_i')$ for every $i \in \mathbb N$. The spheres $S_v(r)$ and $S_w(r)$ are supposed to be equal for $r \geq r_0$ which implies that $d(v,v_r) = d(w,v_r)$ and $d(v,v_r') = d(w,v_r')$ for large enough values of $r$. But then we would have
\[
	d(v,v_r') > d(v,v_r) = d(w,v_r) > d(w,v_r') = d(v,v_r'). \qedhere
\]
\end{proof}

It is known that each graph $G$ has a unique decomposition into prime graphs with respect to the Cartesian product. It is easy to see that, if $G$ is locally finite, then it only has finitely many factors. Hence an infinite, locally finite graph must have at least one infinite prime factor. If there is more than one infinite prime factor, then $G$ can be decomposed into two infinite factors and in this case $G$ will be $2$-distinguishable by Corollary \ref{cor:distinct_spheres}. However, this fact can also be seen as a corollary to Lemma \ref{lem:fixlayers}.

\begin{cor}
\label{cor:infinitefactors}
Let $G$ be a locally finite graph with more than one infinite prime factor. Then a random colouring of $G$ is almost surely distinguishing.
\end{cor}

\begin{proof}
If $G$ has two infinite prime factors then it can be written as  $G = G_1 \sq G_2$ where both $G_1$ and $G_2$ are infinite. Now every vertex is uniquely defined by its  $G_1$-layer and its $G_2$-layer. Both of these layers are almost surely fixed by a random colouring. Hence for every vertex $v$ the probability that  the stabiliser of a random colouring is contained in the stabiliser of $v$ is $1$. 

Since there are only countably many vertices this implies that the stabiliser of a random colouring is almost surely trivial.
\end{proof}

As a direct consequence we get the following result about powers of locally finite graphs.

\begin{cor}
Let $G$ be a Cartesian power of an infinite, locally finite graph. Then a random colouring of $G$ is almost surely distinguishing. {\hfill \qedsymbol}
\end{cor}

Finally, the following result states, that if we would like to prove Conjecture~\ref{con:random}, it suffices to consider prime graphs.

\begin{cor}
If a random colouring is almost surely distinguishing for every locally finite, prime graph with infinite motion, then it is almost surely distinguishing for every locally finite graph with infinite motion.
\end{cor}

\begin{proof}
By Corollary~\ref{cor:infinitefactors} it suffices to consider graphs with only one infinite prime factor. Let $G = G_1 \sq G_2$ be a factorisation of such a graph where $G_1$ is the unique infinite prime factor and let $c$ be a random colouring of $G$. 

By Lemma~\ref{lem:fixlayers} all $G_1$-layers are almost surely setwise fixed by every automorphism in $(\Aut G)_c$. By assumption $c$ is almost surely distinguishing for $G_1$ because $G_1$ is an infinite prime graph. Hence $c$ almost surely fixes every $G_1$-layer pointwise.
\end{proof}

\subsection{Growth bounds}

In the last part of this paper we will be concerned with growth bounds. We say that a graph has growth $f(n)$ if there is a vertex $v_0$ and a constant $c$ such that the ball $B_{v_0}(n)$ has at most cardinality $c f(n)$ for every $n \in \mathbb N$. Notice that $f(n)$ is independent of the choice of $v_0$. Furthermore recall that the sphere around $v_0$ with radius $n$ is defined as $S_{v_0} (n) = \{ v \in V \mid d(v_0,v) = n\}$. Since $B_{v_0}(n) = \bigcup _{k=0}^n S_{v_0}(k)$ it is clear that $\vert S_{v_0}(n) \vert$ fulfills the same growth bound as $\vert B_{v_0}(n) \vert$.

We will now show the following extension of a result in \cite{growth} stating that a graph fulfilling certain growth bounds is $2$-distinguishable. Once again, our contribution is to show that for such graphs a random colouring is almost surely distinguishing.

\begin{thm}
\label{thm:growth}
Let $G$ be a graph with infinite motion and growth $\bigO\left( 2^{(\frac 1 2 - \varepsilon) \sqrt n}\right)$. Then a random colouring of $G$ is almost surely distinguishing.
\end{thm}

In order to prove this result we need two auxiliary lemmas. Firstly, we will use the following refinement of Lemma~\ref{lem:motion}. The proof stated in Section~\ref{sec:randomcolour} also works for this result.

\begin{lem}
\label{lem:motion_prob}
Let $S$ be a finite set and let $\Delta$ be a set of nontrivial permutations of $S$ with motion $\geq m$. Let $c$ be a random colouring of $S$. Then
\[
\prob [\exists \gamma \in \Delta \colon c \varphi = c] \leq \vert \Delta \vert \, 2^{-\frac m 2}. \tag*{\qedsymbol}
\]
\end{lem}

The second auxiliary lemma we will use can be found in \cite{growth}, although the implications stated below have been known before, see for example \cite{cuimle} where they play a crucial role in the proof of one of the main results.

\begin{lem}
Let $G$ be a graph with infinite motion and let $\varphi$ be an automorphism of $G$. Let $V'$ be the set of vertices moved by $\varphi$. Then the subgraph of $G$ induced by $V'$ only has infinite components. {\hfill \qedsymbol}
\end{lem}

In particular the above lemma implies that
\begin{itemize}
\item if an automorphism acts nontrivially on a finite set, then it also acts nontrivially on its boundary,
\item if an automorphism fixes the boundary of a finite set pointwise, then it fixes the whole set pointwise, and
\item if two automorphisms coincide on the boundary of a finite set, then they coincide on the whole set.
\end{itemize}

Now we are ready to prove Theorem~\ref{thm:growth}.

\begin{proof}[Proof of Theorem~\ref{thm:growth}]
Let $c$ be a random coloring of $G$ and choose a vertex $v_0 \in V$. For every $v \in V$ let $\Delta_{v_0}^v$ be the set of automorphisms which map $v_0$ to $v$. Clearly, $(\Delta_{v_0}^v)_{v \in V}$ is a countable decomposition of $\Aut G$. Hence we only need to show that $\Delta_{v_0}^v$ almost surely contains no automorphism $\varphi$ such that $\varphi c =c$.

For $v \nsim_S v_0$ this follows from Lemma~\ref{lem:sim_s}. If $v \sim_S v_0$ then it follows from the following claim:
\begin{itemize}
\item[(\textasteriskcentered)] Let $\Delta_k$ be the set of automorphisms that fix $S_{v_0}(i)$ setwise but not pointwise for every $i \geq k$. Then a random colouring almost surely breaks every automorphism in $\Delta_k$.
\end{itemize}
Assume that (\textasteriskcentered) is true and let $\varphi \in \Delta_{v_0}^v$ for some $v \sim_S V_0$. Since for every $v \sim_S v_0$ there is some index $i$ such that $S_{v_0}(i) = S_v(i)$ those spheres are fixed setwise for $i$ large enough. Furthermore, since $G$ has infinite motion, $\varphi$ has to act nontrivially on infinitely many of the spheres. If it acts nontrivially on some sphere $S_{v_0}(k)$ then it also acts nontrivially on $S_{v_0}(i)$ for each $i > k$. Hence $\varphi$ is contained in some set $\Delta_k$. 

By (\textasteriskcentered) a random colouring $c$ almost surely breaks all of $\Delta_k$ and there are only countably many values for $k$. Hence $c$ almost surely breaks every automorphism in the union of the $\Delta_k$. This implies that a random colouring almost surely breaks all of $\Delta_{v_0}^v$ which completes the proof of the theorem.

So we only need to show that (\textasteriskcentered) holds for every $k$. Let $n > k$. Because of the growth condition on the graph we know that there is some constant $c$ such that
\[
\vert B_{v_0}(n^2) \vert \leq c \, 2^{(\frac 1 2 - \varepsilon) n}. 
\]
This in particular implies that the same upper bound holds for the size of each sphere $S_{v_0}(i)$ for $i < n^2$. For $1 \leq j \leq n-1$ define
\begin{align*}
R_j &= B_{v_0}((j+1)n) \setminus B_{v_0}(jn),\\
\Lambda_j' &= \{\varphi \in \Delta_k \mid \varphi \text{ moves at most } 2^j \text{vertices in some } S_{v_0}(i) \text{ for } i >(j+1)n\},\\
\Lambda_j &= \Lambda_j' \setminus \Lambda_{j-1}'.
\end{align*}
Let $\Pi_j$ be the set of different permutations induced by $\Lambda_j$ on $R_j$. 

The next step is to estimate the probability that a random coloring of $R_j$ breaks all automorphisms in $\Lambda_j$ or, equivalently, all permutations in $\Pi_j$. Since we would like to use Lemma~\ref{lem:motion_prob} we need to establish estimates for the cardinality of $\Pi_j$ and the motion of $\Pi_j$ on $R_j$.

To estimate the number of different permutations observe that two automorphisms that coincide on $S_i$ for some $i >(j+1)n$ also have to coincide on $R_j$. Hence it suffices to estimate the number of permutations on $S_i$ which move less than $2^j$ vertices and add those estimates up. Since the size of $S_i$ is bounded by $2^{(\frac 1 2 - \varepsilon) n}$, the number of such permutations will be bounded by
\[
	\binom{c \, 2^{(\frac 1 2 - \varepsilon) n}}{2^j} (2^j)! \leq \frac{2^{2^j(\frac 1 2 - \varepsilon) n + 2^j \log c}}{(2^j)!} (2^j)! = 2^{2^j(\frac 1 2 - \varepsilon) n + 2^j \log c}.
\]
Adding those estimates up for $(j+1)n \leq i \leq n^2$ we obtain
\[
	\vert \Pi_j \vert \leq n^2 2^{2^j(\frac 1 2 - \varepsilon) n + 2^j \log c}.
\]

In order to estimate the motion $m$ of $\Pi_j$ on $R_j$ observe that an element of $\Lambda_j$ moves at least $2^{j-1}$ vertices in every sphere $S_i$ for $jn < i < (j+1)n$. Otherwise it would be contained in $\Lambda_{j-1}'$. Adding those estimates up we get
\[
	m \geq n 2^{j-1}.
\]

Let $X_j$ denote the event that there is a permutation $\pi \in \Pi_j$ that preserves a random coloring $c$ of $R_j$. Plugging the estimates from above into Lemma~\ref{lem:motion_prob} we obtain
\begin{align*}
\log \prob [X_j] 
& \leq \log \vert \Pi_j \vert -\frac m 2 \\
&\leq 2 \log n + 2^j(\frac 1 2 - \varepsilon) n + 2^j \log c - 2^{j-1} n \\
&= -\varepsilon 2^j n + 2^j \log c + 2 \log n.
\end{align*}
If we choose $n$ large enough this implies that 
\[
\log \prob [X_j] \leq  -\varepsilon 2^{j-1} n \leq -\varepsilon n.
\]

The probability that for every $j$ a random coloring of $R_j$ breaks $\Pi_j$ is now given by
\[
\prod_{j=1}^{n-1} (1-\prob [X_j]) \geq (1-2^{-\varepsilon n})^n
\]
which tends to $1$ as $n$ goes to infinity. Finally observe that if $n$ is large enough then
\[
	\Delta_k = \bigcup_{j=1}^{n-1} \Lambda_j
\]
because the motion on $B_{v_0}(n^2)$ is bounded by the number of vertices in $B_{v_0}(n^2)$. The set $\Lambda_j'$ contains all automorphisms whose motion is at most $2^j$ hence for $n$ large enough and $j \geq \frac n 2$ it will be true that $\Delta_k = \Lambda_j'$.
\end{proof}

\bibliographystyle{abbrv}
\bibliography{sources}

\begin{thebibliography}{10}

\bibitem{MR1394549}
M.~O. Albertson and K.~L. Collins.
\newblock Symmetry breaking in graphs.
\newblock {\em Electron. J. Combin.}, 3(1):Research Paper 18, approx.\ 17 pp.,
  1996.

\bibitem{cuimle}
J.~Cuno, W.~Imrich, and F.~Lehner.
\newblock Distinguishing graphs with infinite motion and nonlinear growth.
\newblock {\em Ars Math. Contemp.}, 7:201--213, 2014.

\bibitem{MR2159259}
R.~Diestel.
\newblock {\em Graph theory}, volume 173 of {\em Graduate Texts in
  Mathematics}.
\newblock Springer-Verlag, Berlin, third edition, 2005.

\bibitem{0635.20001}
D.~M. Evans.
\newblock {A note on automorphism groups of countably infinite structures.}
\newblock {\em Arch. Math.}, 49:479--483, 1987.

\bibitem{MR0335368}
R.~Halin.
\newblock Automorphisms and endomorphisms of infinite locally finite graphs.
\newblock {\em Abh. Math. Sem. Univ. Hamburg}, 39:251--283, 1973.

\bibitem{pre05906416}
R.~Hammack, W.~Imrich, and S.~Klav{\v{z}}ar.
\newblock {\em {Handbook of product graphs. 2nd ed.}}
\newblock {Discrete Mathematics and Its Applications. Boca Raton, FL: CRC
  Press. xviii, 518~p.}, 2011.

\bibitem{MR2302543}
W.~Imrich, S.~Klav{\v{z}}ar, and V.~Trofimov.
\newblock Distinguishing infinite graphs.
\newblock {\em Electron. J. Combin.}, 14(1):Research Paper 36, 12 pp., 2007.

\bibitem{istw}
W.~Imrich, S.~M. Smith, T.~Tucker, and M.~E. Watkins.
\newblock Infinite motion and 2-distinguishability of groups and graphs.
\newblock preprint.

\bibitem{MR0081274}
A.~Karrass and D.~Solitar.
\newblock Some remarks on the infinite symmetric groups.
\newblock {\em Math. Z.}, 66:64--69, 1956.

\bibitem{growth}
F.~Lehner.
\newblock Distinguishing graphs with intermediate growth.
\newblock preprint.

\bibitem{MR0073600}
I.~Maurer.
\newblock Les groupes de permutations infinies.
\newblock {\em Gaz. Mat. Fiz. Ser. A.}, 7:400--408, 1955.

\bibitem{moller}
R.~G. Moller.
\newblock Graphs, permutations and topological groups.
\newblock 2010.
\newblock http://arxiv.org/pdf/1008.3062v2.pdf.

\bibitem{rubin}
F.~Rubin.
\newblock Problem 729.
\newblock {\em J. Recreational Math.}, 11:128, 1979.
\newblock (solution in volume 12, 1980).

\bibitem{MR924157}
W.~Rudin.
\newblock {\em Real and complex analysis}.
\newblock McGraw-Hill Book Co., New York, third edition, 1987.

\bibitem{MR1617449}
A.~Russell and R.~Sundaram.
\newblock A note on the asymptotics and computational complexity of graph
  distinguishability.
\newblock {\em Electron. J. Combin.}, 5:Research Paper 23, 7 pp.\, 1998.

\bibitem{smtuwa}
S.~M. Smith, T.~W. Tucker, and M.~E. Watkins.
\newblock Distinguishability of infinite groups and graphs.
\newblock {\em Electron. J. Combin.}, 19(2):Research Paper 27, 10 pp., 2012.

\bibitem{MR2776826}
T.~W. Tucker.
\newblock Distinguishing maps.
\newblock {\em Electron. J. Combin.}, 18(1):Paper 50, 21, 2011.

\bibitem{MR2302536}
M.~E. Watkins and X.~Zhou.
\newblock Distinguishability of locally finite trees.
\newblock {\em Electron. J. Combin.}, 14(1):Research Paper 29, 10 pp., 2007.

\end{thebibliography}

\end{document}